\documentclass[a4paper,11pt]{amsart}
\usepackage[utf8]{inputenc}
\usepackage{csquotes}
\usepackage{amsmath, amssymb, amsthm, xcolor,enumerate, enumitem}
\usepackage[english]{babel}
\usepackage{tikz-cd}
\usepackage[all]{xy}
\usepackage{microtype}
\usepackage[colorlinks=true, citecolor=blue, linkcolor=blue, bookmarks=true]{hyperref}
\usepackage{bbm}
\usepackage{graphicx}
\usepackage{adjustbox}
\usepackage{comment}
\usepackage{stmaryrd,mathtools,commath,mathrsfs}

\newcounter{theorem}
\newtheorem{theorem}[theorem]{Theorem}
\newtheorem{lemma}[theorem]{Lemma}
\newtheorem{prop}[theorem]{Proposition}

\theoremstyle{definition}
\newtheorem{defn}[theorem]{Definition}

\theoremstyle{remark}
\newtheorem*{remark*}{Remark}
\newtheorem{rmk}[theorem]{Remark}

\numberwithin{equation}{section}
\allowdisplaybreaks

\newcommand{\R}{\mathbb{R}}
\newcommand{\C}{\mathrm{C}^*}
\newcommand{\Z}{\mathcal{Z}}
\newcommand{\id}{\mathrm{id}}

\newcommand{\Aut}{\mathrm{Aut}}
\newcommand{\Ad}{\mathrm{Ad}}

\newcommand{\cO}{\mathcal{O}}

\newcommand{\ontop}[2]{\stackrel{\mathclap{\normalfont\mbox{\footnotesize{#1}}}}{\ensuremath{#2}}} 

\title[Generic classification of quasi-free flows]{Generic classification of the quasi-free flows on the Cuntz algebra $\mathcal{O}_2$}

\author[R.\ Neagu]{Robert Neagu}
\address{\hskip-\parindent Robert Neagu, Department of mathematics, KU Leuven, Celestijnenlaan 200B, 3001, Leuven, Belgium.}
\email{robert.neagu@kuleuven.be}

\thanks{Funded by the European Union. Views and opinions expressed are however those of the author only and do not necessarily reflect those of the European Union or the European Research Council. Neither the EU nor the ERC can be held responsible for them.}

\begin{document}

\begin{abstract}
In this article, we show that the quasi-free flows on the Cuntz algebra $\cO_2$ are generically classifiable by the inverse temperature of their unique KMS state. Along the way, we show that a large class of quasi-free flows on the Cuntz algebra $\cO_n$ are generically equivariantly $\Z$-stable and their corresponding dual action has the Rokhlin property generically.
\end{abstract}

\maketitle

\numberwithin{theorem}{section}	

\section*{Introduction}
\renewcommand*{\thetheorem}{\Alph{theorem}}

Noncommutative dynamics is a central and pervasive theme in operator algebras, largely due to the rich variety of ways in which groups can act on noncommutative structures.
While the study of actions of discrete groups proved to be a fundamental tool for understanding the structure and symmetries of operator algebras, many key applications in geometry and physics center around time evolutions; that is, continuous actions of $\mathbb{R}$, or \emph{flows}. For instance, powered by the groundbreaking theory of Tomita and Takesaki \cite{Tak70}, the study of flows proved to be a quintessential ingredient in the seminal Connes-Haagerup classification of injective factors \cite{ConnesInj, HaagerupFactors}. On the $\C$-algebraic side, the study of flows facilitated important developments such as in the theory of derivations (see for example \cite{Kad66,Ell74}) or in Bratteli–Robinson’s approach to quantum statistical mechanics \cite{BraRob87,BraRob81}.

A notoriously challenging problem is to determine when two flows on an operator algebra are the same up to some appropriate notion of equivalence (usually this is called \emph{cocycle conjugacy}).
One fundamental feature which makes flows particularly difficult to manoeuvre is the potential presence of traces on the induced crossed product. 
Consequently, it is natural to first consider flows with special properties which force the absence of traces on the crossed product. 
Indeed, the archetypal example is the so-called \emph{Rokhlin property} introduced by Kishimoto in \cite{KishRokhlinFlows}.
Among many noteworthy aspects, if $\alpha$ is a flow with the Rokhlin property on a Kirchberg algebra $A$, then the crossed product $A\rtimes_\alpha\mathbb{R}$ is again a Kirchberg algebra. In particular, the crossed product has no traces.

So far, if one looks at peak achievements in this direction such as the Masuda--Tomatsu classification \cite{MaTo16} or Szabó's classification of Rokhlin flows on Kirchberg algebras \cite{SzaRokhlinFlows}, it is evident that general classification results could only be achieved by restricting to the class of \emph{Rokhlin flows}.
Even the recent groundbreaking classification of group actions on Kirchberg algebras due to Gabe--Szabó \cite{GASZ25,DynamicalKP}, when restricted to the case of the group being $\mathbb{R}$, requires the flows to have the Rokhlin property.

This article aims to ignite the classification of flows without the Rokhlin property.
One prominent class of such flows is provided by the quasi-free flows on Cuntz algebras. 
In the case of the Cuntz algebra $\mathcal{O}_2$, assuming rational independence of the defining parameters, Kishimoto proved a dichotomy result for this class of flows: they either have the Rokhlin property or they have a unique KMS state \cite{Kish02}. In the same paper, Kishimoto showed that all Rokhlin quasi-free flows on $\mathcal{O}_2$ are cocycle conjugate to each other and posed the question of whether the cocycle conjugacy class of the ones without the Rokhlin property is determined by the inverse temperature of their unique KMS state (see the comments after \cite[Theorem 1.2]{Kish02}). Our main result provides a partial answer to Kishimoto's question and to \cite[Problem LV]{Questions} and constitutes the first - albeit generic - classification of faithful flows without the Rokhlin property.

\begin{theorem}\label{thm: IntroGenericClassif}[Theorem \ref{thm: GenericClassif}]
Let $L=(L_1,L_2)\in\mathbb{R}^2$ and $\alpha^L$ be the flow on $\cO_2$ defined by $\alpha_t^L(s_k)=e^{itL_k}s_k$ for $k=1,2$ and any $t\in\mathbb{R}$. If $L_1$ and $L_2$ are rationally independent and have the same sign, then $\alpha^L$ is generically classifiable, up to cocycle conjugacy, by the inverse temperature of the unique KMS state with respect to $\alpha^L$.
\end{theorem}

Following the notation in Theorem \ref{thm: IntroGenericClassif}, the strategy is to show that the $\mathbb{R}$-$\C$-algebra $(\cO_2\rtimes_{\alpha^L}\mathbb{R},\widehat{\alpha^L})$ generically falls under the umbrella of \cite[Theorem C]{SzaRokhlinFlows}. In particular, we will classify the corresponding dual actions.
This strategy has often been successfully implemented in the von Neumann setting (see for example \cite{Jones83, JonesTak84, MaTo16}) and, although not stated explicitly in this form, the approach of classifying dual actions instead also appeared in work of Bratteli and Kishimoto. Indeed, they managed to classify certain circle actions on Kirchberg algebras by classifying the dual $\mathbb{Z}$-actions \cite[Corollary 4.1]{BratKish}.

The two main obstacles to apply \cite[Theorem C]{SzaRokhlinFlows} are showing that the crossed product $\cO_2\rtimes_{\alpha^L}\mathbb{R}$ is classifiable and that the dual action $\widehat{\alpha^L}$ has the Rokhlin property. 
First, we will strengthen the result in \cite[Corollary 7.0.3]{Rob12} and show that a large class of quasi-free flows on Cuntz algebras $\cO_n$ are generically equivariantly $\Z$-stable.

\begin{theorem}\label{thm: IntroEquivZStable}[Theorem \ref{thm: EquivZStable}]
\begin{enumerate}[label=\textit{(\roman*)}]
\item The quasi-free flows on $\cO_2$ are generically equivariantly $\Z$-stable.
\item For $p,q\in\mathbb{R}$ and $n\geq 3$, consider the flow $\alpha^{(p,q)}$ on $\mathcal{O}_n$ given by $\alpha^{(p,q)}_t(s_k)=e^{it(p+(k-1)q)}s_k$ for any $t\in\R$ and $k=1,\ldots,n$.
The quasi-free flows $\alpha^{(p,q)}$ are generically equivariantly $\Z$-stable.
\end{enumerate}
\end{theorem}

Showing that the dual actions have the Rokhlin property generically is the most challenging part of the proof of Theorem \ref{thm: IntroGenericClassif}. 
In fact, this holds for more general flows on Cuntz algebras $\cO_n$ with $n\geq 2$, so we will record it separately. 
In the case $n=2$, the following theorem offers a partial answer to \cite[Question 6.17]{SzaRokhlinFlows}.

\begin{theorem}\label{thm: IntroRokhlin}[Theorem \ref{thm: GenericRokh}]
\begin{enumerate}[label=\textit{(\roman*)}]
\item The Rokhlin property is generic among the dual actions induced by the quasi-free flows on $\cO_2$.
\item The Rokhlin property is generic among the dual actions induced by the quasi-free flows $\alpha^{(p,q)}$ on Cuntz algebras $\cO_n$ for $n\geq 3$.
\end{enumerate}
\end{theorem}

To prove the theorem above, we use a duality result from \cite{CNS25} which yields that it is enough to show that the quasi-free flows are generically \emph{pointwise strongly approximately inner} (see Definition \ref{defn: ApproxInner}), thus obtaining a $\C$-counterpart of \cite[Lemma 6.45]{MaTo16}.
The main technical innovation lies in establishing a general framework that subsumes methods introduced by Kishimoto in \cite{Kish02}. Precisely, he shows that the right shift on $\bigotimes_{n\in\mathbb{N}}M_2$ has a Rokhlin-like property by considering the enveloping $\mathbb{Z}$-$\C$-algebra $\bigotimes_{n\in\mathbb{Z}}M_2$ equipped with the right shift. In our more general setting, we show that the canonical endomorphism of $\cO_n$, when restricted to the fixed-point algebras of the quasi-free flows with rationally dependent parameters, has a Rokhlin-like property. Considering the stationary inductive limit of each fixed-point algebra together with the canonical endomorphism, we get a simple, monotracial AF-algebra, together with an automorphism which extends said canonical endomorphism. In analogy with the shift on the CAR algebra (see \cite{BKSR93,KishShifts96,RohlinShiftsKish00}), this automorphism turns out to have the Rokhlin property.

Finally, we would like to point out that Theorems \ref{thm: IntroEquivZStable} and \ref{thm: IntroRokhlin} provide a clear path for generically classifying quasi-free flows on $\cO_n$, provided one extends Szabo's classification result from \cite[Theorem C]{SzaRokhlinFlows}.

\subsection*{Acknowledgements}
The author was supported by the EPSRC grant EP/R513295/1, by the European Research Council under the European Union's Horizon Europe research and innovation programme (ERC grant AMEN-101124789), and by the postdoctoral fellowship 1204626N of the Research Foundation Flanders (FWO).
The author would like to thank Stuart White, Gábor Szabó, and Chris Schafhauser for helpful conversations, and Gábor Szabó for showing him a proof of Lemma \ref{lemma: RokhlinLemmaGabor}.

For the purpose of open access, the author has applied a CC BY public copyright license to any author accepted manuscript version arising from this submission.
\allowdisplaybreaks

\section{Preliminaries}
\numberwithin{theorem}{section}

\subsection{Sequence algebras}

Given a separable $\C$-algebra $A$, we denote by $A_\infty$ the $\C$-algebra obtained as the quotient of the bounded sequences of $A$ by the null sequences. In particular, we will realise $A$ as a subalgebra of $A_\infty$ by viewing elements in $A$ as constant sequences in $A_\infty$. Then, the central sequence algebra of $A$, denoted by $A_\infty\cap A'$ is the set of elements $x\in A_\infty$ such that $xa=ax$ for any $a\in A$. However, if $A$ is non-unital, we will consider the corrected central sequence algebra, as defined by Kirchberg in \cite{KirchbergAbel}. Precisely, let $$A_\infty\cap A^\perp=\{x\in A_\infty\colon xa=ax=0, \ \forall a\in A\}$$ be the two-sided annihilator of $A$ inside $A_\infty$, and define the corrected central sequence algebra by $(A_\infty\cap A')/(A_\infty\cap A^\perp)$. This will be denoted by $F_\infty(A)$.

With the set-up above, let $G$ be a second-countable, locally compact group and $\alpha\colon G\curvearrowright A$ be a continuous action. In this case, we will say that $(A,\alpha)$ is a $G$-$\C$-algebra. Then, we get a (not necessarily continuous) $G$-action $\alpha_\infty$ on $A_\infty$. Similarly, there is an induced action on $F_\infty(A)$ which will be denoted by $\tilde{\alpha}_\infty$. For ease of notation, the induced action on $A_\infty\cap A'$ will still be denoted by $\alpha_\infty$. These actions are in general not continuous, so we restrict to the continuous part when necessary. For instance, we consider $$A_{\infty,\alpha} = \{x \in A_\infty \colon [g \to \alpha_{\infty,g} (x)]\ \text{is norm-continuous}\}.$$ Likewise, we will write $F_{\infty,\alpha}(A)$ for the continuous part of $F_\infty(A)$.

\subsection{Dual actions}

Let $A$ be a $\C$-algebra, $G$ be a second-countable, locally compact, abelian group, and $\alpha\colon G\curvearrowright A$ be a continuous action. For each $\chi\in\widehat{G}$, define a $*$-isomorphism $\widehat{\alpha}_\chi\colon C_c(G,A)\to C_c(G,A)$ by $\widehat{\alpha}_\chi(f)(g) = \overline{\chi(g)}f(g).$
Then $\widehat{\alpha}_\chi$ extends to
an automorphism of the crossed product $A\rtimes_\alpha G$, and we get a continuous homomorphism $\widehat{\alpha}\colon \widehat{G}\to \Aut(A\rtimes_\alpha G).$ We will call $\widehat{\alpha}$ the \emph{dual action}. By Takai's duality \cite{Takai}, we have a canonical equivariant isomorphism \[((A\rtimes_\alpha G)\rtimes_{\widehat{\alpha}}\widehat{G},\widehat{\widehat{\alpha}})\cong (A\otimes\mathcal{K}(L^2(G)), \alpha\otimes\Ad\rho),\]where $\rho$ stands for the right-regular representation of the group $G$. We invite the reader to consult \cite[Chapter 7]{WilliamsBook} for a more detailed exposition.

We will now define an equivalence relation for group actions. In particular, we will show how equivalence of dual actions imply an equivalence of the initial actions.

\begin{defn}\label{defn: CocycleConj}
Let $G$ be a second-countable, locally compact group, and let $(A,\alpha),(B,\beta)$ be $G$-$\C$-algebras. Then $(A,\alpha)$ and $(B,\beta)$ are said to be cocycle conjugate if there exists an isomorphism $\varphi\colon A\to B$ together with a strictly continuous map $\mathbbm{u}\colon G\to\mathcal{U}(\mathcal{M}(B))$ such that $$\mathbbm{u}_{gh}=\mathbbm{u}_g\beta_g(\mathbbm{u}_h)$$ and $$\varphi\circ\alpha_g=\Ad(\mathbbm{u}_g)\circ\beta_g\circ\varphi$$ for any $g,h\in G$. If $\mathbbm{u}_g=1$ for any $g\in G$, then $(A,\alpha)$ and $(B,\beta)$ are said to be conjugate.
\end{defn}

In this article, we aim to obtain cocycle conjugacy results for actions by appealing to cocycle conjugacy of their corresponding dual actions.
We first record the following well-known result (see for example \cite[Corollary 4.2]{BratKish} and \cite[Proposition 2.9]{GardellaUnpublished}).

\begin{prop}\label{prop: CocycleConj}
Let $A,B$ be $\C$-algebras, $G$ be a second-countable, locally compact, abelian group, and $\alpha\colon G\curvearrowright A$ and $\beta\colon G\curvearrowright B$ be continuous actions. Then the following statements hold.
\begin{enumerate}[label=\textit{(\roman*)}]
    \item If $(A,\alpha)$ and $(B,\beta)$ are cocycle conjugate, then $(A\rtimes_\alpha G,\widehat{\alpha})$ and $(B\rtimes_\beta G,\widehat{\beta})$ are conjugate.\label{item: Cocycle1}
    \item If $(A\rtimes_\alpha G,\widehat{\alpha})$ and $(B\rtimes_\beta G,\widehat{\beta})$ are cocycle conjugate, then $(A\otimes\mathcal{K}(L^2(G)),\alpha\otimes\Ad\rho)$ and $(B\otimes\mathcal{K}(L^2(G)),\beta\otimes\Ad\rho)$ are conjugate.\label{item: Cocycle2}
    \item Assume further that $A$ and $B$ are unital, simple, separable, purely infinite and let $p$ be a minimal projection in $\mathcal{K}(L^2(G))$. If $(A\otimes\mathcal{K}(L^2(G)),\alpha\otimes\Ad\rho)$ and $(B\otimes\mathcal{K}(L^2(G)),\beta\otimes\Ad\rho)$ are conjugate via an isomorphism \[\varphi\colon A\otimes\mathcal{K}(L^2(G)) \to B\otimes\mathcal{K}(L^2(G))\] which satisfies $K_0(\varphi)([1_A\otimes p]_0)=[1_B\otimes p]_0$, then $(A,\alpha)$ and $(B,\beta)$ are cocycle conjugate.\label{item: Cocycle3}
\end{enumerate}
\end{prop}

\begin{proof}
A variant of part \ref{item: Cocycle1} appears in \cite[Proposition 2.9]{GardellaUnpublished}. However, since this result is unpublished, we will include the details. Suppose first that there exist an isomorphism $\theta\colon A \to B$ and a one-cocycle $\{\mathbbm{u}_g\}_{g\in G}$ for $\beta$ such that $\Ad(\mathbbm{u}_g)\circ\beta_g\circ\theta=\theta\circ\alpha_g$ for any $g\in G$. Then one can define \[\varphi_0 \colon C_c(G,A)\to C_c(G,B), \  \varphi_0(\zeta)(g) = \theta(\zeta(g))\mathbbm{u}_g,\] for $\zeta\in C_c(G,A)$ and $g\in G$. Since $\theta$ is a cocycle isomorphism, it follows that $\varphi_0$ extends to an isomorphism $\varphi\colon A\rtimes_\alpha G\to B\rtimes_\beta G$. Note that $\varphi$ is precisely obtained by extending $\theta$ and the map which sends each canonical unitary $\lambda_g^\alpha$ to $\mathbbm{u}_g\lambda_g^\beta$. To see that $\varphi$ is the required conjugacy, we will check that $\varphi$ intertwines the dual actions. Given $\zeta\in C_c(G,A)$, $g\in G$, and $\chi\in\widehat{G}$, we get that
\begin{align*}
\varphi(\widehat{\alpha}_\chi(\zeta))(g) =  \theta(\widehat{\alpha}_\chi(\zeta)(g))\mathbbm{u}_g
= \overline{\chi(g)}\varphi(\zeta)(g)
= \widehat{\beta}_\chi(\varphi(\zeta))(g),
\end{align*} which shows that $\varphi$ is a conjugacy from $(A\rtimes_\alpha G,\widehat{\alpha})$ to $(B\rtimes_\beta G,\widehat{\beta})$. 

Then \ref{item: Cocycle2} follows from \ref{item: Cocycle1} using Takai's duality \cite{Takai}. To prove \ref{item: Cocycle3}, we will follow the strategy in \cite[Corollary 4.2]{BratKish}. Let $\varphi$ be an isomorphism from $A\otimes\mathcal{K}(L^2(G))$ to $B\otimes\mathcal{K}(L^2(G))$ such that 
\begin{equation}\label{eq: Cocycle1}
(\beta_g\otimes\Ad\rho_g)\circ\varphi=\varphi\circ (\alpha_g\otimes\Ad\rho_g), \ g\in G.
\end{equation}
Since $K_0(\varphi)([1_A\otimes p]_0)=[1_B\otimes p]_0$ and $B$ is simple and purely infinite, as in \cite[Corollary~4.2]{BratKish}, there exists a unitary $v$ in the multiplier algebra of $B\otimes\mathcal{K}(L^2(G))$ such that \[\Ad(v)(\varphi(1_A\otimes x))=1_B\otimes x, \ \text{for any} \ x\in \mathcal{K}(L^2(G)).\]Thus, there exists an isomorphism $\sigma\colon A\to B$ such that
\begin{equation}\label{eq: Cocycle2}
\Ad(v)\circ\varphi= \sigma\otimes \id_{\mathcal{K}}.
\end{equation}As in the proof of \cite[Corollary 4.2]{BratKish}, we get that 
\begin{align*}
&\Ad(v(\beta_g\otimes\Ad\rho_g)(v^*))(\beta_g\circ\sigma)\otimes \id_{\mathcal{K}} \\ &= \Ad(v)\circ\varphi\circ\varphi^{-1}\circ \Ad((\beta_g\otimes\Ad\rho_g)(v^*))\circ((\beta_g\circ\sigma)\otimes\id_\mathcal{K})\\
&\ontop{\eqref{eq: Cocycle2}}{=} (\sigma\otimes\id_\mathcal{K})\circ\varphi^{-1}\circ(\beta_g\otimes\Ad\rho_g)\circ\Ad(v^*)\circ (\sigma\otimes\Ad\rho_{g^{-1}}) \\
&\ontop{\eqref{eq: Cocycle1}}{=} (\sigma\otimes\id_\mathcal{K})\circ (\alpha_g\otimes\Ad\rho_g)\circ\varphi^{-1}\circ\Ad(v^*)\circ (\sigma\otimes\Ad\rho_{g^{-1}})\\
&\ontop{\eqref{eq: Cocycle2}}{=} (\sigma\otimes\id_\mathcal{K})\circ(\alpha_g\otimes\Ad\rho_g)\circ(\sigma^{-1}\otimes\id_\mathcal{K})\circ (\sigma\otimes\Ad\rho_{g^{-1}})\\
&=(\sigma\circ\alpha_g)\otimes\id_\mathcal{K}.
\end{align*}
Letting $\mathbbm{u}_g= v(\beta_g\otimes\Ad\rho_g)(v^*)(1_B\otimes p)$ for any $g\in G$, we produced a one-cocycle for $\beta$ in the unitary group of $B$ which satisfies\[\sigma\circ\alpha_g=\Ad(\mathbbm{u}_g)\circ\beta_g\circ\sigma.\] Hence $(A,\alpha)$ and $(B,\beta)$ are cocycle conjugate.
\end{proof}

\section{Properties of group actions}

In this section, we will collect a number of regularity properties for group actions which will be used in the sequel. We will primarily focus on integer actions and continuous actions of the real numbers.

\subsection{Integer actions}

\begin{defn}[cf.\ {\cite[Definition 2.1 (ii)]{GSSSA19}}]\label{defn: RokhlinIntegers}
Let $A$ be a separable, unital $\C$-algebra and let $\alpha$ be an
automorphism of $A$. We say that $\alpha$ has the Rokhlin property, if for every $n\in\mathbb{N}$, there exist approximately central sequences\footnote{That is the induced elements $e,f$ are in the central sequence algebra $A_\infty\cap A'$. The Rokhlin property can be defined in the non-unital case too by taking $e,f\in F_\infty(A)$.} of projections $e_k,f_k\in A$ such that 
\[1=\lim\limits_{k\to\infty}\sum\limits_{j=0}^{n-1}\alpha^j(e_k)+\sum\limits_{l=0}^n\alpha^l(f_k).\footnote{See also \cite{RorCuntz93} and \cite[Definition 4.1]{Ror95}.}\]   
\end{defn}

\begin{rmk}\label{rmk: RokhlinProp}
Given $n\in\mathbb{N}$, $\varepsilon>0$, and a finite set $\mathcal{F}\subseteq A$, we will say that two projections $e,f\in A$ are bases for two \emph{$(n,\varepsilon,\mathcal{F})$-approximately Rokhlin towers} if $\|ea-ae\|,\|fa-af\|\leq\varepsilon$ for any $a\in\mathcal{F}$, and $$\left\|\sum\limits_{j=0}^{n-1}\alpha^j(e)+\sum\limits_{l=0}^n\alpha^l(f)-1\right\|\leq\varepsilon.$$
\end{rmk}

We will obtain the Rokhlin property for an integer action by appealing to \emph{equivariant approximate divisibility}. Since we will also need this property for continuous actions of $\mathbb{R}$, we shall define it in the context of second-countable, locally compact groups (see also \cite[Lemma 2.12]{CNS25}).

\begin{defn}\label{defn: EquivApproxDiv}
Let $A$ be a separable $\C$-algebra, $G$ be a second-countable, locally compact group, and $\alpha\colon G\curvearrowright A$ be a continuous action. Then we say that $\alpha$ is equivariantly approximately divisible if there exists a unital $*$-homomorphism $$\theta\colon M_2\oplus M_3\to F_\infty(A)^{\tilde{\alpha}_\infty}.$$   
\end{defn}

In the case when the acting group is the integers, equivariant approximate divisibility can be employed to obtain the Rokhlin property. I would like to thank Gábor Szabó for showing me a proof of the following lemma. Here, $\mathcal{Z}$ stands for the Jiang-Su algebra introduced in \cite{JiangSu99}. A $\C$-algebra $A$ is said to be $\mathcal{Z}$-stable if $A\cong A\otimes\mathcal{Z}$.

\begin{lemma}\label{lemma: RokhlinLemmaGabor}
Let $A$ be a simple, separable, unital, nuclear, $\Z$-stable $\C$-algebra with a unique trace. Let $\alpha\colon\mathbb{Z}\curvearrowright A$ be a strongly outer\footnote{If $\tau$ is the unique trace on $A$, then $\alpha$ is said to be strongly outer if the $\mathbb Z$-action induced by $\alpha$ on $\pi_\tau(A)''$ is outer.} and equivariantly approximately divisible action. Then $\alpha$ has the Rokhlin property.   
\end{lemma}

\begin{proof}
Let $\gamma$ be the noncommutative Bernoulli shift on the infinite tensor product $\Z^{\otimes\mathbb{Z}}$. In particular, $\gamma$ is strongly outer (see for example \cite[Example 3.3]{GSSSA19}).

Since $\alpha$ is equivariantly approximately divisible, there exists a unital $*$-homomorphism $\theta\colon (M_2\oplus M_3)^{\otimes\infty}\to (A_\infty\cap A')^{\alpha_\infty}$ \cite[Corollary 1.13]{KirchbergAbel}. Moreover, $\alpha$ is cocycle conjugate to $\alpha\otimes\gamma$ by \cite[Theorem D]{GSSSA19}. In particular, combining \cite[Theorem C]{GSSSA19} and \cite[Lemma 2.12]{SzaSSA2}, we get a unital and equivariant $*$-homomorphism $\zeta\colon (\Z,\gamma)\to (A_\infty\cap A',\alpha_\infty)$, such that the image of $\zeta$ also commutes with the image of the $*$-homomorphism $\theta$ defined above. Thus, by the universal property of the maximal tensor product, there exists a unital and equivariant $*$-homomorphism $$\eta\colon (\Z\otimes (M_2\oplus M_3)^{\otimes\infty},\gamma\otimes \id)\to (A_\infty\cap A',\alpha_\infty).$$

Let $\varepsilon>0$ and $n\in\mathbb{N}$. By \cite[Theorem~1.3]{Kish95} (see also \cite[Theorem 2.16]{GSSSA19}), $\gamma\otimes\id_{M_{2^\infty}}$ has the Rokhlin property on $\Z\otimes M_{2^\infty}$, and $\gamma\otimes\id_{M_{3^\infty}}$ has the Rokhlin property on $\Z\otimes M_{3^\infty}$. Therefore, there exists $k_0\in\mathbb{N}$ such that for any $k\geq k_0$, there exist $(n,\varepsilon,\{1\})$-approximate Rokhlin towers in $\Z\otimes M_{2^k}$ and in $\Z\otimes M_{3^k}$ (see Remark \ref{rmk: RokhlinProp}). Furthermore, for $k\geq 2k_0$, we have that $$(M_2\oplus M_3)^{\otimes k}\cong \bigoplus\limits_{j=0}^k (M_{2^j}\otimes M_{3^{k-j}})^{\oplus\binom kj},$$ so there exist $(n,\varepsilon,\{1\})$-approximate Rokhlin towers in $\Z\otimes\Z\otimes(M_2\oplus M_3)^{\otimes k}$.

Since $\gamma$ is strongly self-absorbing \cite[Theorem C]{GSSSA19}, considering the image of the $(n,\varepsilon,\{1\})$-approximate Rokhlin towers through $\eta$, there exist projections $e_0,f_0\in A_\infty\cap A'$ such that $$\left\|\sum\limits_{j=0}^{n-1}\alpha_\infty^j(e_0)+\sum\limits_{l=0}^n\alpha_\infty^l(f_0)-1\right\|\leq\varepsilon.$$ By Kirchberg's $\varepsilon$-test \cite[Lemma 1.1]{CETW21} (see also \cite[Lemma A.1]{KirchbergAbel}), we can find projections $e,f\in A_\infty\cap A'$ such that $$\sum\limits_{j=0}^{n-1}\alpha_\infty^j(e)+\sum\limits_{l=0}^n\alpha_\infty^l(f)=1,$$which shows that $\alpha$ has the Rokhlin property.
\end{proof}

\subsection{One-parameter automorphism groups}

A \emph{flow} $\alpha$ on a $\C$-algebra $A$ is a one-parameter group of automorphisms $(\alpha_t)_{t\in\mathbb{R}}$ such that the map $t\mapsto\alpha_t(a)$ is norm-continuous for any $a\in A$. We will often write $\alpha\colon\mathbb{R}\curvearrowright A$ to express that $\alpha$ is a flow on $A$.

One of the most fundamental properties in the study of flows is the \emph{Rokhlin property} defined by Kishimoto in \cite{KishRokhlinFlows}. However, Kishimoto only defined this property for unital $\C$-algebras, so we will record the general definition from \cite{SzaRokhlinFlows}.

\begin{defn}{\cite[Definition 2.8]{SzaRokhlinFlows}}\label{defn: Rokhlin}
Let $A$ be a separable $\C$-algebra and $\alpha\colon\mathbb{R}\curvearrowright A$  be a flow. We say that
$\alpha$ has the Rokhlin property, if for every $p > 0$, there exists a unitary $u \in F_\infty(A)$ such that $\tilde{\alpha}_{\infty,t}(u)= e^{ipt}u$ for all $t \in \mathbb{R}$.
\end{defn}

This property of flows coincides with the \emph{abelian Rokhlin property} for $\mathbb{R}$-actions defined in \cite{CNS25}. In fact, we will crucially use a duality result between the Rokhlin property and a property defined in \cite{CNS25} as \emph{pointwise strong approximate innerness}. For convenience of the reader, we shall record the definition.

\begin{defn}[cf.\ {\cite[Definition 3.1]{CNS25}}]\label{defn: ApproxInner}
Let $A$ be a separable $\C$-algebra and $\alpha\colon \mathbb{R}\curvearrowright A$ be a flow. Then we say that $\alpha$ is pointwise strongly approximately inner if for any $t\in \mathbb{R}$, there exists a sequence of contractions $v_n$ in $A$ such that 
\begin{itemize}
    \item $v_nv_n^*$ and $v_n^*v_n$ converge strictly to $1_{\mathcal{M}(A)}$,
    \item $\alpha_t(a)=\lim\limits_{n\to\infty}v_nav_n^*$ for any $a\in A$, and
    \item $\lim\limits_{n\to\infty}\max\limits_{t\in K}\|\alpha_t(v_n)-v_n\|=0$ for any compact subset $K\subseteq \mathbb{R}$.
\end{itemize}
\end{defn}

Another property which will be of interest in this paper is equivariant $\Z$-stability of flows.

\begin{defn}\label{defn: EquivZStable}
Let $A$ be a separable $\C$-algebra and $\alpha\colon \mathbb{R}\curvearrowright A$ be a flow. Then we say that $\alpha$ is equivariantly $\mathcal{Z}$-stable if $\alpha$ is cocycle conjugate to $\alpha\otimes\id_{\mathcal{Z}}$.   
\end{defn}

\begin{rmk}\label{rmk: EquivZStable}
Similarly to the McDuff-type characterisation of $\mathcal{Z}$-stability in the non-equivariant setting \cite[Theorem 2.2]{TWSSA07}, it is known that a flow $\alpha$ on a separable $\C$-algebra $A$ is equivariantly $\mathcal{Z}$-stable if and only if there exists a unital embedding $\mathcal{Z}\hookrightarrow F_\infty(A)^{\tilde{\alpha}_\infty}$ \cite[Corollary 3.8]{SzaSSA}.
\end{rmk}

We will access equivariant $\Z$-stability through \emph{equivariant approximate divisibility}, as per Definition \ref{defn: EquivApproxDiv}.

\begin{lemma}\label{lemma: ApproxDiv}
Let $A$ be a separable $\C$-algebra and $\alpha\colon\mathbb{R}\curvearrowright A$. If $\alpha$ is equivariantly approximately divisible, then $\alpha$ is equivariantly $\Z$-stable.    
\end{lemma}

\begin{proof}
By definition, there exists a unital $*$-homomorphism $$\theta\colon M_2\oplus M_3\to F_\infty(A)^{\tilde{\alpha}_\infty}.$$ Then, \cite[Corollary 1.13]{KirchbergAbel} shows that there exists a unital $*$-homomorphism $\zeta\colon (M_2\oplus M_3)^{\otimes\infty}\to F_\infty(A)^{\tilde{\alpha}_\infty}.$ By \cite[Corollary 7, Corollary 12]{EllRorAbel}, there exists a unital embedding $j\colon\Z\to (M_2\oplus M_3)^{\otimes\infty}$. Hence $\zeta\circ j\colon\Z\to F_\infty(A)^{\tilde{\alpha}_\infty}$ is a unital $*$-homomorphism, which shows that $\alpha$ is equivariantly $\Z$-stable (Remark \ref{rmk: EquivZStable}).
\end{proof}

\subsection{Quasi-free flows}
For $2\leq n < \infty$, let the Cuntz algebra $\cO_n$ be generated by a family of isometries $s_1,s_2,\ldots, s_n$. Let $\mathcal{H}$ is a Hilbert space of dimension $n\geq 2$.

Following the notation in \cite{Eva80}, define the full Fock space of $\mathcal{H}$ by $F(\mathcal{H})=\bigoplus\limits_{r=0}^\infty\otimes^r\mathcal{H}$, where $\otimes^0\mathcal{H}$ is a one-dimensional Hilbert space spanned by a unit vector $\Omega$. One can then define a linear map $\cO_F\colon \mathcal{H}\to \mathcal{B}(F(\mathcal{H}))$ by $$\cO_F(f)\left(f_1\otimes\ldots\otimes f_r\right)=f\otimes f_1\otimes\ldots\otimes f_r$$ and $\cO_F(f)\left(\Omega\right)= f$ for any $r\in\mathbb{N}$ and any $f,f_1,\ldots,f_r\in\mathcal{H}$. Let $\cO(\mathcal{H})=\cO_F(\mathcal{H})/\mathcal{K}(F(\mathcal{H}))$ be the quotient by the compact operators and define a linear map $\cO\colon\mathcal{H}\to\cO(\mathcal{H})$ by $\cO=\pi\circ\cO_F$, where $\pi\colon\cO_F(\mathcal{H})\to\cO(\mathcal{H})$ is the canonical quotient map. Since $\mathcal{H}$ has dimension $n$, $\cO(\mathcal{H})$ is isomorphic to the Cuntz algebra $\cO_n$ \cite{CuntzAlg}. Furthermore, if $U$ is a
unitary on the Hilbert space $\mathcal{H}$, then there is a unique $*$-isomorphism $\cO(U)$ of $\cO(\mathcal{H})$ such that $\cO(U)\cO(f) = \cO(Uf)$, for any $f\in\mathcal{H}$. In fact, the map $U\to\cO(U)$ is continuous for the corresponding strong topologies because $\|\cO(f)\| = \| f \|$, for any $f\in\mathcal{H}$. Therefore, if $\{U_t\colon t\in\mathbb{R}\}$ is a strongly continuous one-parameter unitary group on $\mathcal{H}$, then $\{\cO(U_t)\colon t\in\mathbb{R}\}$ is a flow on $\cO(\mathcal{H})$, and hence on $\cO_n$. These flows are called \emph{quasi-free flows}.

\begin{rmk}\label{rmk: QuasiFree}
Since $\mathcal{H}$ has dimension $n$, Stone's theorem ensures the existence of a self-adjoint operator $h\in\mathcal{B}(\mathcal{H})$ such that $U_t=e^{ith}$ for any $t\in\mathbb{R}$. Since we can diagonalise $h$ and its eigenvalues are real, up to conjugacy, $U_t$ is a diagonal matrix with diagonal entries of the form $e^{itx}$, where $x\in\mathbb{R}$. Thus, up to conjugacy, any quasi-free flow on the Cuntz algebra $\cO_n$ is determined by a tuple $L=(L_1,L_2,\ldots, L_n)\in\mathbb{R}^n$.
\end{rmk}

Thus, for any tuple $L\in\mathbb{R}^n$, we consider the flow $\alpha^L\colon\R\curvearrowright\cO_n$ given by $\alpha^L_t(s_k)=e^{itL_k}s_k$ for any $1\leq k\leq n$. Throughout the rest of this article, whenever we call a flow quasi-free, we assume it is of the form $\alpha^L$ defined above.

\section{Generic classification}\label{sect: QFreeGeneric}

We shall now restrict our attention to quasi-free flows on the Cuntz algebra $\cO_2$ together with a subclass of quasi-free flows on Cuntz algebras $\mathcal{O}_n$ for $n\geq 3$. In the latter case, for $p,q\in\mathbb{R}$, we consider the flow $\alpha^{(p,q)}$ on $\mathcal{O}_n$ given by $\alpha^{(p,q)}_t(s_k)=e^{it(p+(k-1)q)}s_k$ for any $t\in\R$ and $k=1,\ldots,n$. Let us first recall some properties regarding the fixed-point algebras of the classes of flows described above. If $L=(L_1,L_2)\in\mathbb{R}^2$ and $L_1$ and $L_2$ are rationally independent, then $\cO_2^{\alpha^L}$ is the \emph{GICAR algebra} and it is generated by elements of the form $s_\mu s_\nu^*$, where $\mu$ and $\nu$ are words in $1$'s and $2$'s such that $|\mu|_i=|\nu|_i$ for $i=1,2$ (see for example \cite[Section 3]{Kish02}). Here, $|\mu|_1$ stands for the number of $1$'s that appear in the word $\mu$, while $|\mu|_2$ stands for the number of $2$'s. Recall that the GICAR algebra is the fixed-point algebra of the CAR algebra $M_{2^\infty}$ under the (periodic) $\mathbb{R}$-action given by \[\bigotimes_{m=1}^\infty\Ad(e_{11}^{(m)}+e^{2\pi i t}e_{22}^{(m)}).\] For Cuntz algebras $\cO_n$ with $n\geq 3$, if $p$ and $q$ are rationally independent, then as in \cite[Proposition 3.3]{Kish02}, we have\footnote{The assumption in \cite[Proposition 3.3]{Kish02} that $p$ and $p+(n-1)q$ have different signs does not affect the fixed-point algebra and it is only used to deduce an appropriate version of \cite[Lemma 3.1]{Kish02}.} \[\cO_n^{\alpha^{(p,q)}}\cong (M_{n^\infty})^\gamma,\]where $\gamma$ is the (periodic) $\mathbb{R}$-action given by \[\bigotimes_{m=1}^\infty\Ad(e_{11}^{(m)}+e^{2\pi i t}e_{22}^{(m)}+\ldots+e^{2\pi i(n-1)t}e_{nn}^{(m)}).\]We will often refer to the $\C$-algebra $(M_{n^\infty})^\gamma$ constructed above as the \emph{generalised GICAR algebra}.

If any of the pairs $(L_1,L_2)$ and $(p,q)$ are rationally dependent, the situation is drastically different, but we can still identify the GICAR algebra and the generalised GICAR algebra as canonical subalgebras of the corresponding fixed-point algebras. 

\begin{lemma}\label{lemma: GicarSubalg}
Let $L=(L_1,L_2)\in\mathbb{R}^2$ and $(p,q)\in\mathbb{R}^2$ such that $L_1L_2\neq 0$ and $pq\neq 0$. Consider the flows $\alpha^L\colon\mathbb{R}\curvearrowright\cO_2$ and $\alpha^{(p,q)}\colon\mathbb{R}\curvearrowright\cO_n$ for some fixed $n\geq 3$.\ Then the GICAR algebra is canonically a $\C$-subalgebra of $\cO_2^{\alpha^L}$ and the generalised GICAR algebra is canonically a $\C$-subalgebra of $\cO_n^{\alpha^{(p,q)}}$. 
\end{lemma}

\begin{proof}
Note that the linear span of elements of the form $s_\mu s_\nu^*$ is dense in $\cO_2$, where $\mu$ and $\nu$ are words in $1$ and $2$'s. Then, $s_\mu s_\nu^*$ is in $\cO_2^{\alpha^L}$ if and only if \[L_1|\mu|_1+L_2|\mu|_2=L_1|\nu|_1+L_2|\nu|_2.\]
Consider $\gamma\colon\mathbb T^2\curvearrowright\cO_2$ given by $\gamma_{(z_1,z_2)}(s_i)=z_is_i$ for $i=1,2$ and let \[H\coloneqq\overline{\{(e^{itL_1},e^{itL_2})\mid t\in\mathbb R\}}\subseteq \mathbb T^2\] be a compact subgroup.
By density, we then get that $\cO_2^{\alpha^L}=\cO_2^{\gamma|_H}$.
Since $H$ is compact, we have a conditional expectation $\cO_2\to \cO_2^{\gamma|_H}$ and thus, $\cO_2^{\alpha^L}$ is generated by elements of the form $s_\mu s_\nu^*$ such that $L_1|\mu|_1+L_2|\mu|_2=L_1|\nu|_1+L_2|\nu|_2$. 

If $L_1/L_2\in\mathbb{R}\setminus\mathbb{Q}$, this implies that $|\mu|_i=|\nu|_i$ for $i=1,2$. Therefore, $\cO_2^{\alpha^L}$ is isomorphic to the GICAR algebra in this case. If $L_1/L_2\in\mathbb{Q}$, the closed linear span of elements of the form $s_\mu s_\nu^*$ such that $|\mu|_i=|\nu|_i$ for $i=1,2$ is a $\C$-subalgebra of $\cO_2^{\alpha^L}$. Hence, the GICAR algebra is canonically a $\C$-subalgebra of $\cO_2^{\alpha^L}$ for any $L\in\mathbb{R}$.

The proof for the flow $\alpha^{(p,q)}$ is similar. An element $s_\mu s_\nu^*\in \cO_n^{\alpha^{(p,q)}}$ if and only if \[p\left(\sum_{k=1}^n |\mu|_k\right)+q\left(\sum_{k=1}^n (k-1)|\mu|_k\right)=p\left(\sum_{k=1}^n |\nu|_k\right)+q\left(\sum_{k=1}^n (k-1)|\nu|_k\right).\]If $p$ and $q$ are rationally independent, then $\sum_{k=1}^n |\mu|_k=\sum_{k=1}^n |\nu|_k$ and $\sum_{k=1}^n (k-1)|\mu|_k=\sum_{k=1}^n (k-1)|\nu|_k$, in which case one gets the generalised GICAR algebra. Furthermore, the closed linear span of elements $s_\mu s_\nu^*$ such that $\sum_{k=1}^n |\mu|_k=\sum_{k=1}^n |\nu|_k$ and $\sum_{k=1}^n (k-1)|\mu|_k=\sum_{k=1}^n (k-1)|\nu|_k$ is contained in $\cO_n^{\alpha^{(p,q)}}$ for any nonzero $p,q\in\mathbb{R}$.
\end{proof}

\begin{rmk}
Although Kishimoto extended the results from \cite{Kish02} to more general quasi-free flows on $\cO_n$ \cite{KishOneCocycle}, the generalised GICAR algebra might not be a $\C$-subalgebra of $\cO_n^{\alpha^L}$ for any tuple $L\in\mathbb{R}^n$ where each two entries are rationally dependent.
\end{rmk}

We can now use the lemma above and the proofs of \cite[Proposition 3.2, Proposition 3.3]{Kish02} to show the following. 

\begin{lemma}\label{lemma: CentralSeqO2}
The following statements hold.
\begin{enumerate}[label=\textit{(\roman*)}]
\item Let $L=(L_1,L_2)\in\mathbb{R}^2$ such that $L_1L_2\neq 0$ and let $\alpha^L\colon\mathbb{R}\curvearrowright\cO_2$. Then there exists a unital and equivariant $*$-homomorphism \[\theta\colon(\cO_2,\alpha^L)\to ((\cO_2)_{\infty,\alpha^L}\cap(\cO_2)',\alpha_\infty^L).\]\label{item: Central1}
\item Let $p,q\in\mathbb{R}$ be nonzero, $n\geq 3$, and $\alpha^{(p,q)}\colon\mathbb{R}\curvearrowright\cO_n$. Then there exists a unital and equivariant $*$-homomorphism \[\theta\colon(\cO_n,\alpha^{(p,q)})\to ((\cO_n)_{\infty,\alpha^{(p,q)}}\cap(\cO_n)',\alpha_\infty^{(p,q)}).\]\label{item: Central2}
\end{enumerate}
\end{lemma}

\begin{proof}
We will first prove \ref{item: Central1}. Let $\lambda$ be the canonical endomorphism of $\cO_2$ given by $\lambda(x)=s_1xs_1^*+s_2xs_2^*$, and let 
$u=\sum\limits_{i,j=1}^2s_is_js_i^*s_j^*$ be a unitary in $\cO_2$. In particular, a direct calculation shows that $u\lambda(s_k)=s_k$ for each $k=1,2$. Then, the proof of \cite[Proposition 3.2]{Kish02} produces a sequence $(v_n)_{n\in\mathbb{N}}$ of unitaries in the GICAR algebra such that $$\lim\limits_{n\to\infty}\|u-v_n\lambda(v_n^*)\|=0.$$ Therefore, we have that for each $k=1,2$,
\begin{equation}\label{eq: CentralSeq}
\|\lambda(v_n^*s_k)-v_n^*s_k\|=\|v_n\lambda(v_n^*s_k)-s_k\|\to \|u\lambda(s_k)-s_k\|=0.
\end{equation}
If we let $r_k=[(v_n^*s_k)_{n\in\mathbb{N}}]\in (\cO_2)_\infty$ for $k=1,2$, then $s_1r_ks_1^*+s_2r_ks_2^*=r_k$. Multiplying on the right with $s_1$ shows that $s_1r_k=r_ks_1$, while multiplying on the right with $s_2$ shows that $s_2r_k=r_ks_2$, so $r_k\in (\cO_2)_\infty\cap (\cO_2)'$ for each $k=1,2$. Moreover, a straightforward check shows that $r_1$ and $r_2$ satisfy the $\cO_2$-relations, so we obtain a unital $*$-homomorphism $\theta\colon\cO_2\to (\cO_2)_\infty\cap (\cO_2)'$. By Lemma \ref{lemma: GicarSubalg}, the sequence $(v_n)_{n\in\mathbb{N}}$ consists of unitaries fixed by the flow $\alpha^L$, so $\theta$ is also equivariant. 

The proof of \ref{item: Central2} follows in the same way as for \ref{item: Central1}, using \cite[Proposition 3.3]{Kish02} instead of \cite[Proposition 3.2]{Kish02}. Precisely, the map $\theta$ is given by sending the generators $s_k$ to a sequence with representatives $(v_m^*s_k)_{m\in\mathbb{N}}$, where $v_m$ is in the generalised GICAR algebra. Then, by Lemma \ref{lemma: GicarSubalg}, $v_m\in \cO_n^{\alpha^{(p,q)}}$ for any nonzero $p,q$, which finishes the proof.
\end{proof}

We shall now prove a series of regularity properties for the $\mathbb{R}$-$\C$-algebras $(\cO_2,\alpha^L)$ and $(\cO_n,\alpha^{(p,q)})$ when the pairs $(L_1,L_2)$ and $(p,q)$ consist of rationally dependent entries. In particular, both $\alpha^L$ and $\alpha^{(p,q)}$ are equivariantly $\Z$-stable.

\begin{prop}\label{prop: FixedPtRational}
Let $L=(L_1,L_2)\in\mathbb{R}^2$ such that $L_1/L_2\in\mathbb{Q}_+$, $(p,q)\in\mathbb{R}^2$ such that $p_1/p_2\in\mathbb{Q}_+$, and $n\geq 3$. Then the following statements hold.
\begin{enumerate}[label=\textit{(\roman*)}]
    \item $\cO_2^{\alpha^L}$ and $\cO_n^{\alpha^{(p,q)}}$ are simple, unital, infinite dimensional, monotracial AF-algebras.\label{item: SimpleAF}
    \item The flows $\alpha^L$ and $\alpha^{(p,q)}$ are equivariantly $\Z$-stable.\label{item: ZStable}
    \item For any separable $\alpha_\infty^L$-invariant $\C$-subalgebra $C\subseteq (\mathcal{O}_2)_{\infty,\alpha^L}$, and any separable $\alpha_\infty^{(p,q)}$-invariant $\C$-subalgebra $D\subseteq (\mathcal{O}_n)_{\infty,\alpha^{(p,q)}}$ there exist a unital $*$-homomorphism \[M_2\oplus M_3\to \left((\mathcal{O}_2)_\infty\cap C'\right)^{\alpha_\infty^L}\] and a unital $*$-homomorphism \[M_2\oplus M_3\to \left((\mathcal{O}_n)_\infty\cap D'\right)^{\alpha_\infty^{(p,q)}}.\]In particular, the flows $\alpha^L$ and $\alpha^{(p,q)}$ are equivariantly approximately divisible.\label{item: ApproxDiv}
\end{enumerate}
\end{prop}

\begin{proof}
Since the fixed-point algebra of a flow is invariant under scaling the flow, so are the properties of the flow being equivariantly approximately divisible or equivariantly $\Z$-stable. Replacing the flow $\alpha^L$ with the flow given by $t\mapsto \alpha_{-t}^L$ if necessary, we can assume that $L_1$ and $L_2$ are both positive. Then, suppose that $L_1/L_2=m/n$ for some coprime positive integers $m$ and $n$. Replacing the flow $\alpha^L$ with the flow given by $t\mapsto \alpha_{nt/L_2}^L$ if necessary, we can assume that $L_1$ and $L_2$ are coprime, positive integers. By a similar argument, we can also assume that $p$ and $q$ are coprime positive integers. Hence, $\cO_2^{\alpha^L}$ and $\cO_n^{\alpha^{(p,q)}}$ are simple, unital, monotracial AF-algebras by \cite[Theorem 4.2]{BJO04} (see also \cite[Theorem~3.1]{Dean01}). Furthermore, the GICAR algebra is a $\C$-subalgebra of $\cO_2^{\alpha^L}$ by Lemma \ref{lemma: GicarSubalg}, so $\cO_2^{\alpha^L}$ is infinite dimensional. Likewise, Lemma \ref{lemma: GicarSubalg} also yields that $\cO_n^{\alpha^{(p,q)}}$ is infinite dimensional.

We will now prove \ref{item: ZStable}. First note that \ref{item: SimpleAF} implies that $\cO_2^{\alpha^L}$ and $\cO_n^{\alpha^{(p,q)}}$ are $\Z$-stable \cite[Theorem 5]{JiangSu99}. Then, Lemma \ref{lemma: CentralSeqO2} yields unital $*$-homomorphisms \[\zeta\colon \cO_2^{\alpha^L}\to \left((\cO_2)_\infty\cap(\cO_2)'\right)^{\alpha_\infty^L}\] and \[\tilde{\zeta}\colon \cO_n^{\alpha^{(p,q)}}\to \left((\cO_n)_\infty\cap(\cO_n)'\right)^{\alpha_\infty^{(p,q)}}.\] Since $\cO_2^{\alpha^L}$ and $\cO_n^{\alpha^{(p,q)}}$ are unital and $\Z$-stable, there exist unital embeddings $\iota\colon\Z\hookrightarrow \cO_2^{\alpha^L}$ and $\tilde{\iota}\colon\Z\hookrightarrow \cO_n^{\alpha^{(p,q)}}$. Thus, $\zeta\circ\iota$ and $\tilde{\zeta}\circ\tilde{\iota}$ are unital embeddings which witness the equivariant $\Z$-stability of $\alpha^L$ and $\alpha^{(p,q)}$ respectively (see \cite[Corollary 3.8]{SzaSSA} and Remark \ref{rmk: EquivZStable}).

We shall now prove \ref{item: ApproxDiv}. Let $C\subseteq (\cO_2)_{\infty,\alpha^L}$ be a separable, $\alpha_\infty^L$-invariant $\C$-subalgebra. Combining Lemma \ref{lemma: CentralSeqO2} and \cite[Lemma 2.13]{SzaSSA2}, it follows that there exists a unital and equivariant $*$-homomorphism $\tilde{\theta}\colon (\cO_2,\alpha^L)\to \left((\cO_2)_{\infty,\alpha^L}\cap C',\alpha_\infty^L\right)$. This then restricts to a unital $*$-homomorphism \[\theta\colon \cO_2^{\alpha^L}\to \left((\cO_2)_\infty\cap C'\right)^{\alpha_\infty^L}.\] By the same argument, there exists a unital $*$-homomorphism $\tilde{\theta}\colon \cO_n^{\alpha^{(p,q)}}\to \left((\cO_n)_\infty\cap D'\right)^{\alpha_\infty^{(p,q)}}$.
Then, by \ref{item: SimpleAF} and \cite[Proposition~4.1]{BlaKumRor92} $\cO_2^{\alpha^L}$ and $\cO_n^{\alpha^{(p,q)}}$ are approximately divisible, so there exist unital $*$-homomorphisms $M_2\oplus M_3\to \cO_2^{\alpha^L}$ and $M_2\oplus M_3\to \cO_n^{\alpha^{(p,q)}}$.
This finishes the proof of \ref{item: ApproxDiv}.
\end{proof}

As a consequence of Proposition \ref{prop: FixedPtRational}, we get that the quasi-free flows on Cuntz algebras are generically equivariantly $\Z$-stable.

\begin{theorem}\label{thm: EquivZStable}
\begin{enumerate}[label=\textit{(\roman*)}]
\item The set of $\{L_2/L_1\in\mathbb{R}_+\setminus\mathbb{Q}\mid L_1,L_2\in\mathbb{R}\}$ such that the flow $\alpha^L$ is equivariantly approximately divisible is a dense subset of $\mathbb{R}_+\setminus\mathbb{Q}$ of second Baire category. In particular, $\alpha^L$ is generically equivariantly $\Z$-stable.\label{item: O2}
\item The set of $\{q/p\in\mathbb{R}_+\setminus\mathbb{Q}\mid p,q\in\mathbb{R}\}$ such that the flow $\alpha^{(p,q)}$ on $\cO_n$ is equivariantly approximately divisible is a dense subset of $\mathbb{R}_+\setminus\mathbb{Q}$ of second Baire category. In particular, $\alpha^{(p,q)}$ is generically equivariantly $\Z$-stable.\label{item: On}
\end{enumerate}
\end{theorem}

\begin{proof}
The proof of \ref{item: On} follows exactly like the proof of \ref{item: O2}, so we will only write the details for \ref{item: O2}. 
Since the property of being equivariantly approximately divisible is invariant under scaling the flow $\alpha^L$, we can assume that $L_1=1$. If $L_2\in\mathbb{Q}_+$, then $\alpha^L$ is equivariantly approximately divisible by \ref{item: ApproxDiv} of Proposition \ref{prop: FixedPtRational}.

Then, for each finite set $\mathcal{F}\subseteq \cO_2$, $\varepsilon>0$, $T>0$, and any $q\in \mathbb{Q}_+$, there exists a neighbourhood $V(q,\mathcal{F},\varepsilon,T)$ such that we have approximate generators of $M_2\oplus M_3$ which approximately commute with $\mathcal{F}$ and are approximately fixed by $\alpha_t^L$ for any $L=(1,L_2)$, where $L_2\in V(q,\mathcal{F},\varepsilon,T)$ and any $|t|\leq T$. Precisely, for any $L_2\in V(q,\mathcal{F},\varepsilon, T)$, there exist positive contractions $(e_{i,j})_{i,j=1}^2, (f_{l,k})_{l,k=1}^3\subseteq \cO_2$ such that for any $1\leq i,j,i',j'\leq 2$ and any $1\leq l,k,l',k'\leq 3$
\begin{itemize}
    \item $\max\{\|e_{i,j}-e_{i,j}^*\|,\|f_{l,k}-f_{l,k}^*\|\}\leq\varepsilon$,
    \item $\max\{\|e_{i,j}e_{i',j'}-\delta_{j,i'}e_{i,j'}\|, \|f_{l,k}f_{l',k'}-\delta_{k,l'}f_{l,k'}\|\}\leq\varepsilon$,
    \item $\|e_{i,j}f_{k,l}\|,\|f_{k,l}e_{i,j}\|\leq\varepsilon$,
    \item $\max\{\|e_{i,j}x-xe_{i,j}\|,\|f_{k,l}x-xf_{k,l}\|\}\leq\varepsilon$, for any $x\in\mathcal{F}$,
    \item $\max_{|t|\leq T}\|\alpha_t^{(1,L_2)}(e_{i,j})-e_{i,j}\|\leq\varepsilon$, and
    \item $\max_{|t|\leq T}\|\alpha_t^{(1,L_2)}(f_{k,l})-f_{k,l}\|\leq\varepsilon$.
\end{itemize}

Let \[G(\mathcal{F},\varepsilon, T)=\bigcup_{q\in \mathbb{Q}_+} V(q,\mathcal{F},\varepsilon, T),\] which is a dense open set in $\mathbb{R}_+$. Letting $\mathcal{F}(\mathcal{O}_2)$ be the set of finite sets of $\cO_2$, we set \[G=\bigcap_{\mathcal{F}\in\mathcal{F}(\cO_2)}\bigcap_{n\in\mathbb{N}}G(\mathcal{F},1/n,n).\]Then $G$ is contained in the set of $L_2\in\mathbb{R}_+$ such that $\alpha^L$ is equivariantly approximately divisible, and by the Baire category theorem, $G$ is a dense $G_\delta$ subset of $\mathbb{R}_+$ containing $\mathbb{Q}_+$. In particular, $G\setminus\mathbb{Q}$ is a dense subset of $\mathbb{R}_+\setminus\mathbb{Q}$ of second Baire category. That $\alpha^L$ is generically equivariantly $\Z$-stable now follows from Lemma \ref{lemma: ApproxDiv}.
\end{proof}

\begin{rmk}\label{rmk: EquivZStableFlows}
If $L_1$ and $L_2$ are rationally independent and have different signs, Kishimoto showed in \cite{Kish02} that $\alpha^L$ has the Rokhlin property. Combined with Szabó classification of Rokhlin flows on Kirchberg algebras \cite{SzaRokhlinFlows}, one gets that $\alpha^L$ is equivariantly $\Z$-stable.
Likewise, combining the main result of \cite{KishOneCocycle} and \cite{SzaRokhlinFlows}, we can also obtain equivariant $\Z$-stability for a large class of quasi-free flows on $\cO_n$. 
In particular, if $p,q$ are rationally independent and $p$ and $p+(n-1)q$ have different signs, then $\alpha^{(p,q)}$ is equivariantly $\Z$-stable.
\end{rmk}

Recall that the strategy for proving Theorem \ref{thm: IntroGenericClassif} is to show that the dual $\mathbb{R}$-action has the Rokhlin property generically. We shall first prove that the dual actions of rational quasi-free flows have this property.
We will start with a series of preparatory lemmas.
Let $L=(L_1,L_2)\in\mathbb{R}^2$ such that $L_1,L_2$ are coprime natural numbers and $\lambda$ be the canonical endomorphism on $\cO_2$, i.e. $\lambda(x)=s_1xs_1^*+s_2xs_2^*$ for any $x\in\mathcal{O}_2$.
Since $\lambda$ restricts to an injective unital endomorphism of $\mathcal{O}_2^{\alpha^L}$, we can consider the stationary inductive limit of $\mathcal{O}_2^{\alpha^L}$ with $\lambda$ as connecting map.
This construction yields a unital $\C$-algebra $A$ together with an automorphism $\gamma$ such that $\mathcal{O}_2^{\alpha^L}$ can be identified with a $\C$-subalgebra of $A$ and $\gamma$ restricts to $\lambda$ on $\mathcal{O}_2^{\alpha^L}$. Moreover, since $\mathcal{O}_2^{\alpha^L}$ is a simple, unital, infinite dimensional AF-algebra with unique trace (Proposition \ref{prop: FixedPtRational}), so is $A$ (see \cite[Theorem 8.2.8]{Strung21} for simplicity of $A$). If $L_1=L_2=1$, one can identify the pair $(\cO_2^{\alpha^L},\lambda)$ with $\bigotimes_{m\in\mathbb{N}}M_2$ equipped with the right shift, while the pair $(A,\gamma)$ becomes $\bigotimes_{m\in\mathbb{Z}}M_2$ equipped with the right shift.

\begin{lemma}\label{lemma: StrOuterAction}
The $\mathbb Z$-action $\gamma$ on $A$ is strongly outer.
\end{lemma}

\begin{proof}
This follows as in the case of the canonical shift on the CAR algebra (see the introduction of \cite{BKSR93}).
Fix a free ultrafilter $\omega$ on $\mathbb N$.
Then, it suffices to show that 
\begin{equation}\label{eq: StrOuter1}
\lim_{n\to\omega}\tau(a\gamma^n(b))=\tau(a)\tau(b), \ a,b\in A.
\end{equation}Indeed, we can combine \cite[Theorem 4.3.22]{BraRob87} and \cite[Theorem 4.3.20]{BraRob87} to conclude that the automorphism $\gamma''$ is ergodic and hence outer. 

Let $a,b\in A$ be contractions and $\varepsilon>0$. Then, there exist $k,l\in\mathbb{N}$ and $x,y\in\cO_2^{\alpha^L}$ contractions such that 
\begin{equation}\label{eq: StrOuter2}
\|\gamma^k(a)-x\|\leq\varepsilon/2 \ \text{and} \ \|\gamma^l(b)-y\|\leq\varepsilon/2.
\end{equation}
Moreover, the trace $\tau$ is $\gamma$-invariant, so 
\[\lim_{n\to\omega}\tau(a\gamma^n(b))=\lim_{n\to\omega}\tau(\gamma^k(a)\gamma^{n}(\gamma^l(b))).\]
For the rest of this argument, if $c,d\in \mathbb C$, we will write $c=_\varepsilon d$ to mean that $|c-d|\leq\varepsilon$.
Suppose now that \eqref{eq: StrOuter1} holds for $x,y$ and use \eqref{eq: StrOuter2} repeatedly to get that 
\begin{align*}
\lim_{n\to\omega}\tau(a\gamma^n(b)) &=_{\varepsilon}\lim_{n\to\omega}\tau(x\gamma^n(y))\\
&= \tau(x)\tau(y)\\
&=_{\varepsilon}\tau(\gamma^k(a))\tau(\gamma^l(b))\\
&=\tau(a)\tau(b).
\end{align*} Since $\varepsilon$ was arbitrary, we get that \eqref{eq: StrOuter1} holds for $a,b$ as long as it holds for any two elements in $\cO_2^{\alpha^L}$.

Thus, it suffices to show that \eqref{eq: StrOuter1} holds for any $a,b\in\cO_2^{\alpha^L}$.
Since $\cO_2^{\alpha^L}$ is the closed linear span of (certain) elements of the form $s_\mu s_\nu^*$ and \eqref{eq: StrOuter1} is linear in each variable, we can assume that $a=s_\mu s_\nu^*$ and $b=s_\eta s_\xi ^*$, where $\mu,\nu,\eta,\xi$ are words in $1$'s and $2$'s.
Then \[a\gamma^n(b)=\sum_{|\alpha|=n}s_\mu s_\nu^* s_\alpha s_\eta s_\xi^* s_\alpha^*,\]where $|\alpha|$ is the length of the word $\alpha$.
Assuming $n$ is large enough, $s_\nu^* s_\alpha=0$ unless $\alpha = (\nu\alpha_0)$ for some word $\alpha_0$ with $|\alpha_0|=n-|\nu|$. 
Therefore
\begin{equation}\label{eq: words}
a\gamma^n(b)=\sum_{|\alpha_0|=n-|\nu|}s_\mu s_{\alpha_0} s_\eta s_\xi^* s_{\alpha_0}^* s_\nu^*.
\end{equation}

Let us fix $\beta\in\mathbb R$ such that \[e^{-\beta L_1}+e^{-\beta L_2}=1,\]and assume that $\mu=\nu$ and $\eta=\xi$.
Then for $n>|\mu|$, by \cite[Theorem~4.2]{BJO04}, \eqref{eq: words} gives 
\begin{align*}
&\tau(a\gamma^n(b))= \sum_{|\alpha_0|=n-|\mu|}e^{-\beta(L_1|\mu\alpha_0\eta|_1+L_2|\mu\alpha_0\eta|_2)}\\
&=\sum_{|\alpha_0|=n-|\mu|}e^{-\beta(L_1|\mu\eta|_1+L_2|\mu\eta|_2)}e^{-\eta(L_1|\alpha_0|_1+L_2|\alpha_0|_2)}\\
&=e^{-\beta(L_1|\mu\eta|_1+L_2|\mu\eta|_2)}(e^{-\beta L_1}+e^{-\beta L_2})^{n-|\mu|}\\
&=e^{-\beta(L_1|\mu\eta|_1+L_2|\mu\eta|_2)},
\end{align*}where we have used the binomial theorem in the third equality.
On the other hand, using \cite[Theorem~4.2]{BJO04}, we also have 
\begin{align*}
    &\tau(a)\tau(b)=\tau(s_\mu s_\mu^*)\tau(s_\eta s_\eta^*)= e^{-\beta(L_1|\mu_1|+L_2|\mu_2|)}e^{-\beta(L_1|\eta_1|+L_2|\eta_2|)}\\
    &= e^{-\beta(L_1|\mu\eta|_1+L_2|\mu\eta|_2)},
\end{align*}and hence $\tau(a\gamma^n(b))=\tau(a)\tau(b)$ for any $n>|\mu|$.

Suppose now that either $\mu\neq \nu$ or $\eta\neq\xi$. 
Then $\tau(a)\tau(b)=0$ by \cite[Theorem 4.2]{BJO04}.
On the other hand, since $(\mu\alpha_0\eta)\neq (\nu\alpha_0\xi)$, combining \cite[Theorem 4.2]{BJO04} and \eqref{eq: words}  yields that $\tau(a\gamma^n(b))=\tau(a)\tau(b)=0$ for any $n>|\nu|$ and hence \eqref{eq: StrOuter1} holds.
Thus, $\gamma$ is strongly outer.
\end{proof}

\begin{lemma}\label{lemma: EquivApproxDivGamma}
The $\mathbb Z$-action $\gamma$ on $A$ is equivariantly approximately divisible.
\end{lemma}

\begin{proof}
We claim that it suffices to find a unital $*$-homomorphism \[M_2\oplus M_3\to \left((\mathcal{O}_2^{\alpha^L})_\infty\cap (\cO_2^{\alpha^L})'\right)^{\lambda_\infty}.\] By the definition of the inductive limit construction, $\mathcal{O}_2^{\alpha^L}\subseteq A$, and $\gamma$ extends $\lambda$, so there exists a canonical $*$-homomorphism $\psi\colon \left((\mathcal{O}_2^{\alpha^L})_\infty\right)^{\lambda_\infty}\to A_\infty^{\gamma_\infty}$.
Thus, it remains to show that if $x\in \left((\mathcal{O}_2^{\alpha^L})_\infty\cap (\cO_2^{\alpha^L})'\right)^{\lambda_\infty}$ and $a\in A$, then $\psi(x)a=a\psi(x)$. 
Take $x=(x_n)_{n\in\mathbb{N}}$, $y\in\cO_2^{\alpha^L}$, and $k\in\mathbb{N}$ such that $a=[(0,\ldots,0,y,\lambda(y),\lambda^2(y),\ldots)]$, where the first $k-1$ entries are equal to $0$.
Since $\psi(x)=(z_n)_{n\in\mathbb{N}}$ with \[z_n=[(x_n,\lambda(x_n),\lambda^2(x_n),\ldots)],\]we get that
\begin{align*}
    \|\psi(x)a-a\psi(x)\|&=\limsup_{n\to\infty}\|z_na-az_n\|\\
    &= \limsup_{n\to\infty}\lim_{m\to\infty}\|\lambda^m(x_n)\lambda^{m-k+1}(y)-\lambda^{m-k+1}(y)\lambda^m(x_n)\|\\
    &= \limsup_{n\to\infty}\lim_{m\to\infty}\|\lambda^{m-k+1}(\lambda^{k-1}(x_n)y-y\lambda^{k-1}(x_n))\|\\
    &= \limsup_{n\to\infty}\|\lambda^{k-1}(x_n)y-y\lambda^{k-1}(x_n)\|\\
    &= \limsup_{n\to\infty}\|x_ny-yx_n\|
\end{align*} since $\|\lambda(x_n)-x_n\|\to 0$. Furthermore, we also know that $\|x_ny-yx_n\|\to 0$, so $\psi(x)a=a\psi(x)$. Hence, $\psi$ restricts to a unital $*$-homomorphism \[\left((\mathcal{O}_2^{\alpha^L})_\infty\cap (\cO_2^{\alpha^L})'\right)^{\lambda_\infty}\to (A_\infty\cap A')^{\gamma_\infty}.\]

Moreover, an element $x\in (\mathcal{O}_2^{\alpha^L})_\infty$ is fixed by $\lambda_\infty$ if and only if $x\in (\mathcal{O}_2^{\alpha^L})_\infty\cap \mathcal{O}_2'$, so it suffices to find a unital $*$-homomorphism \[M_2\oplus M_3\to (\mathcal{O}_2^{\alpha^L})_\infty\cap (\mathcal{O}_2)'.\] Noticing that $\alpha^L$ is a periodic flow, it follows that $(\cO_2^{\alpha^L})_\infty\cong (\cO_2)_\infty^{\alpha_\infty^L}$ \cite[Lemma 3.12]{SeqSplit}, and hence the existence of such a $*$-homomorphism follows from \ref{item: ApproxDiv} of Proposition \ref{prop: FixedPtRational}. 
\end{proof}

\begin{theorem}\label{thm: RokhlinRationalQFree}
Let $L=(L_1,L_2),(p,q)\in\mathbb{R}^2$ such that $L_1/L_2\in\mathbb{Q}_+$, $p/q\in\mathbb{Q}_+$, and $n\geq 3$.
\begin{enumerate}[label=\textit{(\roman*)}]
\item The dual flow $\widehat{\alpha^L}$ on $\cO_2\rtimes_{\alpha^L}\mathbb{R}$ has the Rokhlin property.\label{item: RokhO2} 
\item The dual flow $\widehat{\alpha^{(p,q)}}$ on $\cO_n\rtimes_{\alpha^{(p,q)}}\mathbb{R}$ has the Rokhlin property.\label{item: RokhOn} 
\end{enumerate}
\end{theorem}

\begin{proof}
By \cite[Corollary 3.8]{CNS25}, the conclusion is equivalent to showing that the flows $\alpha^L$ and $\alpha^{(p,q)}$ are pointwise strongly approximately inner. Since this property is clearly invariant under scaling the flow, we can assume that $(L_1,L_2)$ and $(p,q)$ are pairs of coprime natural numbers. We will first show that $\alpha^L$ is pointwise strongly approximately inner. 

Let $t\in\mathbb{R}$ and let \[u_t=e^{itL_1}s_1s_1^*+e^{itL_2}s_2s_2^*\] be a unitary in $\mathcal{O}_2^{\alpha^L}$. We claim that for any $m\in\mathbb{N}$, there exists a unitary $v_{t,m}\in \mathcal{O}_2^{\alpha^L}$ such that 
\begin{equation}\label{eq: Rokh}
\|u_t-v_{t,m}\lambda(v_{t,m}^*)\|\leq\frac{1}{m},
\end{equation}
where $\lambda(x)=s_1xs_1^*+s_2xs_2^*$ for any $x\in\mathcal{O}_2$. Note that if we prove this claim, we have that \[\alpha_t^L(s_k)=u_ts_k=\lim\limits_{m\to\infty}v_{t,m}\lambda(v_{t,m}^*)s_k=\lim\limits_{m\to\infty}v_{t,m}s_kv_{t,m}^*, \ k=1,2.\]Hence, $\alpha^L$ is pointwise strongly approximately inner.

Recall from the discussion preceeding Lemma~\ref{lemma: StrOuterAction} that considering the stationary inductive limit of $\mathcal{O}_2^{\alpha^L}$ with $\lambda$ as connecting map yields a unital $\C$-algebra $A$ together with an automorphism $\gamma$ such that $\mathcal{O}_2^{\alpha^L}$ can be identified with a $\C$-subalgebra of $A$ and $\gamma$ restricts to $\lambda$ on $\mathcal{O}_2^{\alpha^L}$.
Moreover, $A$ is a simple, unital, infinite dimensional AF-algebra with unique trace.
Combining Lemma~\ref{lemma: StrOuterAction}, Lemma~\ref{lemma: EquivApproxDivGamma}, and Lemma~\ref{lemma: RokhlinLemmaGabor}, $\gamma$ has the Rokhlin property.

We will now prove \eqref{eq: Rokh}.
Since $\gamma$ has the Rokhlin property, combining \cite[Proposition 2.4]{RohlinShiftsKish00} and \cite[Proposition 2.2]{RohlinShiftsKish00} yields that for every $m\in\mathbb{N}$ there exists a unitary $\tilde{v}_{t,m}\in A$ such that 
\begin{equation}\label{eq: Rokh1}
\|u_t-\tilde{v}_{t,m}\gamma(\tilde{v}_{t,m}^*)\|\leq\frac{1}{2m}.
\end{equation}Then, by construction of $A$ and $\gamma$, there exists $k\in\mathbb{N}$ and a unitary $w_{t,m}\in \mathcal{O}_2^{\alpha^L}$ such that
\begin{equation*}
\|\gamma^k(\tilde{v}_{t,m})-w_{t,m}\|\leq\frac{1}{4m}.
\end{equation*}Therefore, we have that
\begin{equation}\label{eq: Rokh2}
\|\gamma^k(\tilde{v}_{t,m})\gamma^{k+1}(\tilde{v}_{t,m}^*)-w_{t,m}\gamma(w_{t,m}^*)\|\leq\frac{1}{2m}.
\end{equation}

Combining \eqref{eq: Rokh1}, \eqref{eq: Rokh2}, and the fact that $\gamma$ is equal to $\lambda$ when restricted to $\mathcal{O}_2^{\alpha^L}$, yields that 
\begin{equation}\label{eq: Rokh4}
\|\lambda^k(u_t)-w_{t,m}\lambda(w_{t,m}^*)\|\leq \frac{1}{m}. 
\end{equation}If\[u_{t,k}=u_t\lambda(u_t)\ldots\lambda^{k-1}(u_t),\]then $u_t=u_{t,k}\lambda^k(u_t)\lambda(u_{t,k}^*)$. Therefore, \eqref{eq: Rokh4} becomes
\[\|u_{t,k}^*u_t\lambda(u_{t,k})-w_{t,m}\lambda(w_{t,m}^*)\|\leq\frac{1}{m}.\]Letting $v_{t,m}=u_{t,k}w_{t,m}$ gives \eqref{eq: Rokh}.

To prove \ref{item: RokhOn}, the argument follows in the same fashion, but considering \[u_t=\sum_{k=1}^n e^{it(p+(k-1)q)}s_ks_k^* \ \text{and} \ \lambda(x)=\sum_{k=1}^n s_kxs_k^*.\] Since $\cO_n^{\alpha^{(p,q)}}$ is also a simple, unital, monotracial AF-algebra, and the flow $\alpha^{(p,q)}$ is equivariantly approximately divisible, the relevant results from Proposition \ref{prop: FixedPtRational} and \cite[Theorem 4.2]{BJO04} are enough to finish the argument in the same way as we did in \ref{item: RokhO2}.
\end{proof}

Combining \cite[Corollary~7.0.3]{Rob12} and Theorem \ref{thm: RokhlinRationalQFree}, the following theorem provides a partial answer to \cite[Question 6.17]{SzaRokhlinFlows}.

\begin{theorem}\label{thm: GenericRokh}
\begin{enumerate}[label=\textit{(\roman*)}]
\item The set of $\{L_2/L_1\in\mathbb{R}_+\setminus\mathbb{Q}\mid L_1,L_2\in\mathbb{R}\}$ such that the dual flow $\widehat{\alpha^L}$ on $\cO_2\rtimes_{\alpha^L}\mathbb{R}$ has the Rokhlin property is a dense subset of $\mathbb{R}_+\setminus\mathbb{Q}$ of second Baire category.\label{item: GenRokO2}
\item The set of $\{q/p\in\mathbb{R}_+\setminus\mathbb{Q}\mid p,q\in\mathbb{R}\}$ such that the flow  $\widehat{\alpha^{(p,q)}}$ on $\cO_n\rtimes_{\alpha^{(p,q)}}\mathbb{R}$ has the Rokhlin property is a dense subset of $\mathbb{R}_+\setminus\mathbb{Q}$ of second Baire category.\label{item: GenRokOn}
\end{enumerate}
\end{theorem}

\begin{proof}
By \cite[Corollary 3.8]{CNS25}, the conclusion of \ref{item: GenRokO2} is equivalent to showing that the set of $\{L_2/L_1\in\mathbb{R}_+\mid L_1,L_2\in\mathbb{R}\}$ such that $\alpha^L$ is pointwise strongly approximately inner is a dense $G_\delta$ subset of $\mathbb{R}_+$ containing $\mathbb{Q}_+$.

Since the property of being pointwise strongly approximately inner is invariant under scaling the flow, we can assume that $L_1=1$. Combining \ref{item: RokhO2} of Theorem \ref{thm: RokhlinRationalQFree} and \cite[Corollary 3.8]{CNS25}, for any finite set $\mathcal{F}\subseteq \mathcal{O}_2$, $\varepsilon>0$, $T>0$, and  $q\in\mathbb{Q}_+$, there exists a neighbourhood $V(q,\mathcal{F},\varepsilon,T)$ of $q\in\mathbb{Q}_+$ such that for any $L_2\in V(q,\mathcal{F},\varepsilon,T)$, and any $|t|\leq T$, there exists a unitary $u_t\in\mathcal{O}_2$ such that \[\|\alpha_t^L(x)-u_txu_t^*\|\leq\varepsilon, \ x\in\mathcal{F}, \ \text{and}\] \[\|\alpha_s^L(u_t)-u_t\|\leq\varepsilon, \ |s|\leq T,\] where $L=(1,L_2)$. 

Let \[G(\mathcal{F},\varepsilon, T)=\bigcup_{q\in\mathbb{Q}_+} V(q,\mathcal{F},\varepsilon, T),\] which is a dense open set in $\mathbb{R}_+$. Letting $\mathcal{F}(\mathcal{O}_2)$ be the set of finite sets of $\cO_2$, we set \[G=\bigcap_{\mathcal{F}\in\mathcal{F}(\cO_2)}\bigcap_{n\in\mathbb{N}}G(\mathcal{F},1/n,n).\]Then $G$ is contained in the set of $L_2\in\mathbb{R}_+$ such that $\alpha^L$ is pointwise strongly approximately inner, and by the Baire category theorem, $G$ is a dense $G_\delta$ subset of $\mathbb{R}_+$ containing $\mathbb{Q}_+$. In particular, $G\setminus\mathbb{Q}$ is a dense subset of $\mathbb{R}_+\setminus\mathbb{Q}$ of second Baire category. 

The proof of \ref{item: GenRokOn} follows exactly the same argument as above, but using \ref{item: RokhOn} of Theorem \ref{thm: RokhlinRationalQFree} instead of \ref{item: RokhO2} of Theorem \ref{thm: RokhlinRationalQFree}.
\end{proof}

\begin{rmk}
If $L_1$ and $L_2$ are rationally independent and have different signs, Kishimoto showed in \cite{Kish02} that $\alpha^L$ has the Rokhlin property. Combined with \cite[Theorem 4.6]{KishInvApproxInner}, one gets that $\widehat{\alpha^L}$ has the Rokhlin property.
Likewise, combining the main result of \cite{KishOneCocycle} and \cite[Theorem 4.6]{KishInvApproxInner}, we can also obtain the Rokhlin property for a large class of dual actions.\ In particular, if $p,q$ are rationally independent and $p$ and $p+(n-1)q$ have different signs, then $\widehat{\alpha^{(p,q)}}$ has the Rokhlin property.
\end{rmk}

We can now combine Theorem \ref{thm: EquivZStable} and Theorem \ref{thm: GenericRokh} to obtain the main result of this article.

\begin{theorem}\label{thm: GenericClassif}
The set of $\{L_2/L_1\in \mathbb{R}_+\setminus\mathbb{Q}\mid L_1,L_2\in\mathbb{R}\}$ such that the flow $\alpha^L$ on $\cO_2$ is classified, up to cocycle conjugacy, by the unique real number $\beta$ given by \[e^{-\beta L_1}+e^{-\beta L_2}=1\] is a dense subset of $\mathbb{R}_+\setminus\mathbb{Q}$ of second Baire category.
\end{theorem}

\begin{proof}
It is enough to prove the statement assuming $L_1$ and $L_2$ are both positive. The strategy is to show that the dual flow $\widehat{\alpha^L}$ falls under the umbrella of \cite[Theorem C]{SzaRokhlinFlows}.
Since $L_1$ and $L_2$ are rationally independent, \cite[Theorem 4.1]{KishKumj96} shows that the crossed product $\cO_2\rtimes_{\alpha^L}\mathbb{R}$ is a simple and stably projectionless $\C$-algebra which has a unique trace up to scalar multiple. In addition, 
\begin{equation}\label{eq: Scale}
\tau\circ\widehat{\alpha^L}_t=e^{-\beta t}\tau, \ \text{where} \ e^{-\beta L_1}+e^{-\beta L_2}=1.
\end{equation} Furthermore, the $K$-theory groups of the crossed $\cO_2\rtimes_{\alpha^L}\mathbb{R}$ are $0$ by Connes' Thom isomorphism \cite{ConnesThom}. Since $\cO_2\rtimes_{\alpha^L}\mathbb{R}$ satisfies the UCT \cite[22.3.5 (g)]{blackadar}, $KK(\cO_2\rtimes_{\alpha^L}\mathbb{R},\cO_2\rtimes_{\alpha^L}\mathbb{R})=0$.
Since $\Z$-stability implies finite nuclear dimension \cite{CE20}, combining Theorem \ref{thm: EquivZStable}, Theorem \ref{thm: GenericRokh}, and \cite[Theorem C]{SzaRokhlinFlows}, we get that the dual flow is generically classifiable, up to cocycle conjugacy, by the real number $\beta$ in the statement of the theorem. Thus, since $K_0(\cO_2\otimes\mathcal{K})=0$, the conclusion follows by combining \ref{item: Cocycle2} and \ref{item: Cocycle3} of Proposition \ref{prop: CocycleConj}.
\end{proof}

\begin{rmk}
The real number $\beta$ in the statement of the theorem above is precisely the inverse temperature of the unique KMS state of $\alpha^L$ on $\cO_2$. Thus, Theorem \ref{thm: GenericClassif} offers a partial answer to a question asked by Kishimoto (see the comments following \cite[Theorem 1.2]{Kish02}).  
\end{rmk}

\begin{rmk}
Due to $K$-theoretic issues, we cannot generalise Theorem \ref{thm: GenericClassif} to Cuntz algebras $\cO_n$ for $n>2$. Precisely if we consider a quasi-free flow $\alpha^{(p,q)}$ on a Cuntz algebra $\cO_n$ for $n>2$, then $K_1(\cO_n\rtimes_\alpha\mathbb{R})\neq 0$, so the classification result in \cite[Theorem C]{SzaRokhlinFlows} cannot be applied. We expect that the techniques developed in this paper will be effective in classifying more general flows (such as quasi-free flows on more general Cuntz algebras) once a classification of Rokhlin flows on a larger class of stably finite $\C$-algebras will be available. 
\end{rmk}

\bibliographystyle{abbrv}
\bibliography{quasifree}
\end{document}